\documentclass[leqno,a4paper,11pt]{amsart}

\usepackage[arrow, matrix, curve]{xy}  
\usepackage{amsmath}
\usepackage{amssymb}
\usepackage{amsthm}
\usepackage{amsfonts}                  
\usepackage{dsfont}                    
\usepackage{graphicx}                  
\usepackage{footmisc}                  
\usepackage{color}
\usepackage[utf8]{inputenc}            
\usepackage{tikz}                      
\usepackage{hyperref}
\usepackage{comment}
\usepackage[ngerman,USenglish]{babel} 

\newtheorem*{thm-plain}{Theorem}

\newtheorem{thm}{Theorem}[section]
\newtheorem{lem}[thm]{Lemma}
\newtheorem{cor}[thm]{Corollary}
\newtheorem{prp}[thm]{Proposition}

\theoremstyle{definition}

\newtheorem{rem}[thm]{Remark}
\newtheorem{nte}[thm]{Note}
\newtheorem{exl}[thm]{Example}

\DeclareMathOperator{\Char}{Char}

\DeclareMathOperator{\PFD}{PFD}
\DeclareMathOperator{\Res}{Res}

\renewcommand{\mathbf}{\mathbb}

\numberwithin{equation}{section}

\begin{document}

\begin{abstract}
  We prove a closed character formula for the symmetric powers \( S^N V(\lambda) \)
  of a fixed irreducible representation \(V(\lambda)\) of a complex semi-simple Lie
  algebra \(\mathfrak{g}\) by means of partial fraction decomposition.
  The formula involves rational functions in rank of \(\mathfrak{g}\) many variables
  which are easier to determine than the weight multiplicities of \( S^N V(\lambda) \)
  themselves. We compute those rational functions in some interesting cases.
  Furthermore, we introduce a residue-type generating function for the weight
  multiplicities of \( S^N V(\lambda) \) and explain the connections between our 
  character formula, vector partition functions and iterated partial fraction
  decomposition.
\end{abstract}  

  \pagestyle{plain}
  \title{ A closed character formula for symmetric powers of irreducible representations}
  \author{Stavros Kousidis$^\dagger$}
  \address{\selectlanguage{ngerman}Stavros Kousidis, Mathematisches Institut, Universit\"at zu K\"oln, Weyertal 86-90, 50931 K\"oln, Germany}
  \email{\href{mailto:skousidi@math.uni-koeln.de}{skousidi@math.uni-koeln.de}}
  \thanks{$^\dagger$Supported by the Max Planck Institute for Mathematics and the Deutsche Forschungsgemeinschaft SFB/TR12 Symmetries and Universality in Mesoscopic Systems.}
  \maketitle
  \tableofcontents

  \pagenumbering{arabic}
  
  \section{Notation}
  \label{sec:notation}

  Let \( \mathfrak{g} \) be a complex semi-simple Lie algebra of rank \( r \).
  Fix a Borel \( \mathfrak{b} \) and a Cartan subalgebra \( \mathfrak{h} \)
  in \( \mathfrak{g} \) and let \( Q = \bigoplus_{i=1}^r \mathbb{Z} \alpha_i \)
  and \( X = \bigoplus_{i=1}^r \mathbb{Z} \omega_i \) be the corresponding root and
  weight lattice spanned by the simple roots and fundamental weights respectively.
  Let \( \alpha_1^{\vee},\ldots,\alpha_r^{\vee} \) be the simple coroots and \( W \)
  the Weyl group. An irreducible representation of \( \mathfrak{g} \) of highest weight
  \( \lambda \in X^+ \), where \( X^+ \) stands for all dominant weights, is denoted by
  \( V(\lambda) \). Its character will be written as \( \Char V(\lambda) \) and it is
  well-known that it is an element of \( \mathbb{Z}[X] \), the integral group ring
  associated to the weight lattice. Each generator \( e^{\mu} \in \mathbb{Z}[X] \)
  yields a function on \( \mathfrak{h}_{\mathbb{R}} \), the real span of
  the simple coroots, by \( x \mapsto e^{\langle \mu,x \rangle} \). In this sense we have
  the associated Fourier series of the character of \( V(\lambda) \) as a function of
  \( \mathfrak{h}_{\mathbb{R}} \), i.e. \( \Char V(\lambda)(ix) = \sum_{\mu \in X}
  m_{\mu} e^{ i \langle \mu , x \rangle } \). To simplify notation in what follows we
  define \( q = e^{i\langle \cdot , x \rangle} \), i.e. \( q^{\mu} = e^{i\langle \mu , x \rangle} \).
  Then, with respect to the coordinate system \( \{ \alpha_1^{\vee}, \ldots , \alpha_r^{\vee} \} \)
  of \( \mathfrak{h}_{\mathbb{R}} \) we have \( q = (q_1,\ldots,q_r) \) with 
  \( q_i = e^{i \langle \cdot , x_i \alpha_i^{\vee} \rangle} \) and for
  \( \mu = c_1 \omega_1 + \ldots + c_r \omega_r \)
  \begin{align*}
    q^{\mu}
      = (q_1,\ldots,q_r)^{(c_1,\ldots,c_r)}
      = q_1^{c_1} \cdots q_r^{c_r}
      \in \mathbb{Z}[q_1^{\pm 1}, \ldots , q_r^{\pm 1}]
      .
  \end{align*} 
  Note, whenever we write \( \mathbb{N} \) we mean the non-negative integers
  \( \{0,1,2,\ldots \} \).

  \section{Introduction and the Main Theorem}
  \label{sec:introduction}
  
  Let \( m_{\lambda,N} : X \rightarrow \mathbb{N} \) be the weight multiplicity function
  for the \(N\)-th symmetric power of a fixed irreducible representation \(V(\lambda)\)
  of \( \mathfrak{g} \), i.e.
  \begin{align*}
    \Char S^N V(\lambda)
      = \sum\limits_{\nu \in X} m_{\lambda,N}(\nu) e^{\nu}
      \in \mathbb{Z}[X]
      .
  \end{align*}
  Then, we have the following combinatorial identity
  \begin{align*}
    m_{\lambda,N}(\nu)
      = \sum\limits_{\genfrac{}{}{0pt}{}{\{\nu_1,\ldots,\nu_N\} \subset X}{\nu_1 + \ldots + \nu_N = \nu}}
          m_{\lambda,1}(\nu_1) \cdots m_{\lambda,1}(\nu_N)
	.
  \end{align*}  
  In general it is a non-trivial problem to determine \(m_{\lambda,N}\). That is, to
  establish a formula depending on \(N\) that counts the unordered pairs
  \( \{\nu_1,\ldots,\nu_N\} \) subject to the restriction \( \nu_1 + \ldots + \nu_N = \nu \). \\
  
  We will instead identify the Fourier series associated to the character of \( S^N V(\lambda) \)
  as an element of \( \mathbb{C}(q_1,\ldots,q_r)[X] \)
  (see section \ref{sec:notation} for the notation). The key point is that this identification
  involves data (apart from terms in \(N\)) which is easier to determine than the function
  \(m_{\lambda,N}\) and depends only on the fixed representation \(V(\lambda)\). 
  Starting point will be Molien's formula (see \cite[Chapter 9, §4.3]{MR2265844}) which
  identifies the graded character of the symmetric algebra of \(V(\lambda)\) as a product of
  geometric series. We will state this result here for a quick reference. 

  \begin{lem}[\protect{compare \cite[Chapter 9, §4.3]{MR2265844}}]
    \label{lem:Molien-formula}
    \begin{align*}
      \Char S V(\lambda) 
        =
        \sum\limits_{N=0}^{\infty} z^N \Char S^N V(\lambda)
	=
	\prod\limits_{\nu \in X}
	\frac{1}{(1-e^{\nu}z)^{\dim V(\lambda)_{\nu}}}
    \end{align*}
  \end{lem}
  
  \noindent Our main result will be Theorem \ref{thm:character-formula} in Section
  \ref{sec:closed-character-formula}. That is,

  \begin{thm-plain}[Character formula]
    Let \( \mathfrak{g} \) be a semi-simple complex Lie algebra of rank \(r\) and \(V(\lambda)\)
    a fixed irreducible representation of \(\mathfrak{g}\) with weight space decomposition
    \( V(\lambda) = \bigoplus_{\nu \in X} V(\lambda)_{\nu} \) and weight multiplicity function
    \( m_{\lambda} : X \rightarrow \mathbb{N}\). Then, with
    \( q = e^{i\langle \cdot , x \rangle} = (q_1 , \ldots , q_r) \) as above, we have
    \begin{align*}
      \Char S^N V(\lambda)(ix)
        =
	\sum\limits_{\nu \in X} q^{N\nu}
	\sum\limits_{k=1}^{m_{\lambda}(\nu)}
        A_{\nu,k}(q) \cdot p_k (N)
	\in \mathbb{C}(q_1,\ldots,q_r)[X]
    \end{align*}
    with rational functions \( A_{\nu,k}(q) \in \mathbb{C}(q_1,\ldots,q_r) \)
    and polynomials \( p_k (N) \in \mathbb{Q}[N] \) of degree \(k-1\) given by 
    \begin{align*}
      p_k(N) = \binom{N+k-1}{N}
      .
    \end{align*}	    
    Furthermore, for a weight
    \( \mu \in X \) and \( l = 0,\ldots,m_{\lambda}(\mu)-1 \) we have
    \begin{align*}
      A_{\mu,m_{\lambda}(\mu) - l}(q)
        =
	\frac{(-1)^l}{l! q^{l \mu}} \cdot \frac{d^l}{(dz)^l}
	\left[ \prod\limits_{\nu \in X \setminus \mu}
	\frac{1}{(1-q^{\nu}z)^{m_{\lambda}(\nu)}} \right]_{z = q^{-\mu}} 
	.
    \end{align*}	    
  \end{thm-plain}

  \noindent We will apply this theorem to prove character formulas in some interesting
  cases, involving in particular concrete expressions for the rational functions.\\
  To the authors' knowledge there is no formula of such type known so far although the
  derivation of the Main Theorem is based on simple
  observations\footnote{and on ``Mickey Mouse''-analysis as Alan Huckleberry has put it
  to me in private communication}.  \\

  In Section \ref{sec:residue-type-generating-function-multiplicities}, Proposition
  \ref{prp:genf-multiplicities} we will prove an integral expression for the generating
  function associated to the weight multiplicity functions \(m_{N,\lambda} \) 
  (evaluated at a fixed weight \( \mu \in X \)) of the sequence of representations 
  \( S^N V(\lambda) \).
  Based on this identity and our Main Theorem above we will explain the nature of this
  generating function and in particular why it is of residue-type.\\

  Section \ref{sec:vector-partition-functions} will be a short sketch of the connections
  between the results of section \ref{sec:closed-character-formula},
  \ref{sec:residue-type-generating-function-multiplicities} and vector partition functions
  and iterated partial fraction decomposition (see e.g. \cite{MR2096740} and \cite{arx09121131}).\\

  Section \ref{sec:Weyl-group-orbits} comments on an important continuation of the present
  discussion. That is, the character formula established in the Main Theorem can be split
  into individual parts belonging to the Weyl group orbits of dominant weights.
  The question is what can be expected from the iterated partial fraction decomposition of those
  individual terms. We illustrate a possible answer by an example. A detailed treatment
  will appear in the full version of this extended abstract.

  \section{A closed character formula for symmetric powers}
  \label{sec:closed-character-formula}

  We will derive a closed character formula for the representation \( S^N V(\lambda) \)
  in terms of a basis of weight vectors of the irreducible representation \( V(\lambda) \)
  with weight multiplicity function \( m_\lambda \) and the parameter \(N\).
  The term ``closed'' will be explained in detail in Note \ref{nte:closed-character-formula}
  once we have proven our main result, Theorem \ref{thm:character-formula}.
  The method we use is the partial fraction decomposition. That is, consider the identity of
  Lemma \ref{lem:Molien-formula} for \( q = e^{ i \langle \cdot , x \rangle } \),
  \( x \in \mathfrak{h}_{\mathbb{R}} \),
  \begin{align}
    \label{eq:graded-character-symmetric}
    \sum\limits_{N=0}^{\infty} z^N \Char S^NV(\lambda)(ix)
      =
      \prod\limits_{\nu \in X}
      \frac{1}{(1-q^{\nu}z)^{m_{\lambda}(\nu)}}
      .
  \end{align}
  Partial fraction decomposition with respect to the variable \(z\) (abbreviated by \( \PFD_z\))
  of the right-hand side of Equation \eqref{eq:graded-character-symmetric} gives

  \begin{prp}
    Let \( \mathfrak{g} \) be a semi-simple complex Lie algebra of rank \(r\) and \(V(\lambda)\)
    a fixed irreducible representation of \(\mathfrak{g}\) with weight space decomposition
    \( V(\lambda) = \bigoplus_{\nu \in X} V(\lambda)_{\nu} \) and weight multiplicity function
    \( m_{\lambda} : X \rightarrow \mathbb{N}\). With
    \( q = e^{i\langle \cdot , x \rangle} = (q_1 , \ldots , q_r) \) as above,
    \begin{align}
      \label{eq:pfd-graded-character}
      \PFD_z \left( \prod\limits_{\nu \in X} \frac{1}{(1-q^{\nu}z)^{m_{\lambda}(\nu)}} \right)
        =
	\sum\limits_{\nu \in X} \sum\limits_{k=1}^{m_{\lambda}(\nu)}
        A_{\nu,k}(q) \frac{1}{(1-q^{\nu}z)^k}
    \end{align}
    where for each \( \nu \in X \) and \( k \in \mathbb{N} \) we have 
    \begin{align*}
      A_{\nu,k}(q)
        \in
        \mathbb{C}(q_1,\ldots,q_r)
	.
    \end{align*}  
  \end{prp}

  \begin{proof}
	  See e.g. \cite{MR533571}, \cite{MR1878556}, \cite[Lemma 1]{arx09121131}.	  
  \end{proof}


  \begin{nte}
    The right-hand side of Equation \eqref{eq:pfd-graded-character} is a finite sum
    since the second summation gives zero if a weight \(\nu\) does not appear in \(V(\lambda)\),
    i.e. \( m_{\lambda}(\nu) = 0 \).
  \end{nte}  

  We aim at a power series expansion of the right-hand side of Equation 
  \eqref{eq:pfd-graded-character} with respect to the variable \(z\). The following proposition
  will make life easier.
 
  \begin{prp}
    \label{prp:iterated-summation-powers}
    For \( \nu \in X \) and \( q = e^{i\langle \cdot , x \rangle} \) as above, we have
    \begin{align*}
      \frac{1}{(1-q^{\nu}z)^k}
        =
	\sum\limits_{N=0}^{\infty} z^N q^{N\nu} p_{k}(N)
    \end{align*}  
    where \( p_{k}(N) \) is a polynomial in \( N \) of degree
    \( k - 1 \) given by
    \begin{align*}
      p_k(N) = \binom{N+k-1}{N}
      .
    \end{align*}
  \end{prp}  

  \begin{proof}
    Write down the Cauchy product of the \(k\)-th power of the geometric series
    \( (1-q^{\nu}z)^{-1} \). Then, you see that \( p_{k}(N) \) is given by
    \begin{align*}
      p_{k}(N)
        =
	\sum\limits_{j_{k-1} = 0}^N \sum\limits_{j_{k-2} = 0}^{j_{k-1}}
        \cdots \sum\limits_{j_{1} = 0}^{j_{2}} 1
	= \binom{N+k-1}{N}
	.
    \end{align*}
  \end{proof}  

  As a direct consequence of Equation \eqref{eq:pfd-graded-character} and Proposition
  \ref{prp:iterated-summation-powers} we obtain our main result.

  \begin{thm}[Character formula]
    \label{thm:character-formula}
    Let \( \mathfrak{g} \) be a semi-simple complex Lie algebra of rank \(r\) and \(V(\lambda)\)
    a fixed irreducible representation of \(\mathfrak{g}\) with weight space decomposition
    \( V(\lambda) = \bigoplus_{\nu \in X} V(\lambda)_{\nu} \) and weight multiplicity function
    \( m_{\lambda} : X \rightarrow \mathbb{N}\). Then, with
    \( q = e^{i\langle \cdot , x \rangle} = (q_1 , \ldots , q_r) \) as above, we have
    \begin{align*}
      \Char S^NV(\lambda)(ix)
        =
	\sum\limits_{\nu \in X} q^{N\nu}
	\sum\limits_{k=1}^{m_{\lambda}(\nu)}
        A_{\nu,k}(q) \cdot p_k (N)
	\in \mathbb{C}(q_1,\ldots,q_r)[X]
    \end{align*}
    with rational functions \( A_{\nu,k}(q) \in \mathbb{C}(q_1,\ldots,q_r) \)
    and polynomials \( p_k (N) \in \mathbb{Q}[N] \) of degree \(k-1\) given by 
    \begin{align*}
      p_k(N) = \binom{N+k-1}{N}
      .
    \end{align*}
    Furthermore, for a weight \( \mu \in X \) and \( l = 0,\ldots,m_{\lambda}(\mu)-1 \)
    we have
    \begin{align}
    \label{eq:rational-functions-differentation}
      A_{\mu,m_{\lambda}(\mu) - l}(q)
        =
	\frac{(-1)^l}{l! q^{l \mu}} \cdot \frac{d^l}{(dz)^l}
	\left[ \prod\limits_{\nu \in X \setminus \mu}
	\frac{1}{(1-q^{\nu}z)^{m_{\lambda}(\nu)}} \right]_{z = q^{-\mu}} 
	.
    \end{align}	    
  \end{thm}

  \begin{proof}
    From Equation \eqref{eq:pfd-graded-character} we see that
    \begin{align*}
      \Char S^NV(\lambda) (ix)
        =
	\Res_{z=0} \left[ \frac{1}{z^{N+1}}
	\sum\limits_{\nu \in X} \sum\limits_{k=1}^{m_{\lambda}(\nu)}
        A_{\nu,k}(q) \frac{1}{(1-q^{\nu}z)^k} \right]
	.
    \end{align*}
    Then, Proposition \ref{prp:iterated-summation-powers} finishes the first part of
    the proof.  For the second part multiply the right-hand side of Equation
    \eqref{eq:pfd-graded-character} by \( (1-q^{\mu}z)^{m_{\lambda}(\mu)} \) which is
    equivalent to take the product over \(X \setminus \mu\) in Equation
    \eqref{eq:rational-functions-differentation}.
    By the product rule of differentiation we see that all summands except the
    \( \mu \)-th one give zero after differentiation and evaluation at \( q^{-\mu} \).
    Therefore, the remaining part is
    \begin{align*}
      & \frac{d^l}{(dz)^l}
        \left[
        \sum\limits_{k=1}^{m_{\lambda}(\mu)}
        A_{\mu,k}(q) (1-q^{\nu}z)^{m_{\lambda}(\mu) - k}
        \right]_{z=q^{-\mu}}
      \\
      & = \left[
        \sum\limits_{k=m_{\lambda}(\mu) -l}^{m_{\lambda}(\mu)}
        A_{\mu,k}(q) (-1)^l q^{l\mu}
	\prod\limits_{i=0}^{l}(m_{\lambda}(\mu) -k -i)(1-q^{\nu}z)^{m_{\lambda}(\mu) - k -l}
        \right]_{z=q^{-\mu}}
      \\
      & = A_{\mu,m_{\lambda}(\mu)-l}(q) (-1)^l q^{l\mu} l!
      .
    \end{align*}	    
  \end{proof}
  
  \begin{nte}
    \label{nte:closed-character-formula}
    Now we are able to explain the term ``closed'' used in the beginning of this
    section. Namely, the identity stated in Theorem \ref{thm:character-formula} 
    shows that all relevant data needed to describe the character of \(S^N V(\lambda)\),
    in particular the rational functions \( A_{\nu,k}(q) \), depends on the weight space
    decomposition and weight multiplicity function \(m_{\lambda}\) of the fixed
    representation \(V(\lambda)\).
  \end{nte}

  \begin{nte}
  \label{nte:multiplicity-1-evaluation}
    Equation \eqref{eq:rational-functions-differentation} might be a simple observation but it is
    a very effective method to compute the rational functions associated to weights of
    multiplicity \(1\). Then, we have no differentiation but just simple evaluation.
    In particular, one can immediately compute the character of the symmetric powers of a
    multiplicity free irreducible representation \(V\). Note that in this case one could also
    obtain the character of \(S^N V\) by plugging the \(k\)-many weights of the representation
    \(V\) into the complete homogeneous symmetric polynomial identity
    \begin{align*}
      h_N (x_1,\ldots,x_k) = \sum\limits_{i=1}^k \frac{x_i^N}{\prod_{j \neq i} (1-x_jx_i^{-1})}.
    \end{align*}
  \end{nte}  

  As a consequence of Note \ref{nte:multiplicity-1-evaluation} we can prove concrete
  character formulas for the symmetric powers of the irreducible representations \(V(m)\) of
  \( \mathfrak{g} \) being of type \( A_1 \) and furthermore for the symmetric powers of
  the fundamental representation \(V(\omega_1)\) of \( \mathfrak{g} \) of type \( A_r \).

  \begin{cor}
    \label{cor:character-formula-sl2_irreducible}
    For \( \mathfrak{g} = \mathfrak{sl}(2,\mathbb{C}) \) and its irreducible representation
    \( V(m) \), \( m \in \mathbb{N} \), the Fourier series of the character of \( S^N V(m) \) is
    given by
    \begin{align*}
      \Char S^N V(m)(ix)
        =
	\sum\limits_{i=0}^{m} q^{N(m-2i)} A_{m-2i,1}(q)
	\in \mathbb{C}(q)[X]
    \end{align*}
    where \( q = e^{ix} \) as above and with rational functions
    \begin{align*}
      A_{m-2i,1}(q)
        =
        (-1)^i q^{(m-i)(m-i+1)}
	\prod\limits_{\genfrac{}{}{0pt}{}{j=0}{j \neq i}}^m \frac{1}{q^{2 |i-j|}-1}
	.
    \end{align*}
  \end{cor}

  \begin{proof}
    The weights of \( V(m) \) are given by \( (m-2i)\omega_1 \) where \( i = 0,\ldots,m \).
    By Theorem \ref{thm:character-formula} we immediately obtain the claimed character
    formula and
    \begin{align*}
      A_{m-2i,1}(q)
        & =
	  \frac{(-1)^0}{0! q^{0(m-2i)\omega_1}} \frac{d^0}{(dz)^0}
	  \left[
	    \prod\limits_{\genfrac{}{}{0pt}{}{j=0}{j \neq i}}^m
	    \frac{1}{1-q^{(m-2j)} z}
	  \right]_{z=q^{-(m-2i)}}
      \\  
        & =
	  \prod\limits_{\genfrac{}{}{0pt}{}{j=0}{j \neq i}}^m
	  \frac{1}{1-q^{(m-2j)} q^{-(m-2i)}}
      \\  
        & =
	  \prod\limits_{0 \leq j < i} \frac{1}{1-q^{2(i-j)}}
	  \prod\limits_{i < j \leq m} \frac{q^{2(j-i)}}{q^{2(j-i)}-1}
      \\  
        & =
	  (-1)^i q^{(m-i)(m-i+1)}
	  \prod\limits_{\genfrac{}{}{0pt}{}{j=0}{j \neq i}}^m \frac{1}{q^{2 |i-j|}-1}
	  .
    \end{align*}
  \end{proof}

  \begin{exl}
  \label{exl:character-formula_sl2_V2}	  
    Let \( \mathfrak{g} = \mathfrak{sl}(2,\mathbb{C}) \). Since \( S^N V(0) = V(0) \) and
    \( S^N V(1) = V(N) \), the first non-trivial example is given by the adjoint representation
    \( V(2) \) and its symmetric powers \( S^N V(2) \).
    By Corollary \ref{cor:character-formula-sl2_irreducible} we have
    \begin{align*}
      \Char S^N V(2)(ix) =
        & 
	\frac{q^6}{(q^4-1)(q^2-1)} \cdot q^{2N} + \frac{-q^2}{(q^2-1)^2} \cdot q^{0} \\
        & + \frac{1}{(q^4-1)(q^2-1)} \cdot q^{-2N}
	.
    \end{align*}
  \end{exl}
  
  \begin{cor}
    \label{cor:character-formula-slr+1_fundamental}
    Let \( \mathfrak{g} = \mathfrak{sl}(r+1,\mathbb{C}) \) and consider its fundamental
    representation \( V(\omega_1) \). Set \( \omega_0 = \omega_{r+1} = 0 \), i.e. extend
    \( q = e^{i \langle \cdot , x \rangle} = (q_1, \ldots , q_r) \) by \( q_0 = q_{r+1} = 1 \).
    Then, 
    \begin{align*}
      \Char S^N V(\omega_1)(ix)
        =
	\sum\limits_{i=0}^r q_i^{-N} q_{i+1}^{N} A_{-\omega_i + \omega_{i+1},1}(q)
	\in \mathbb{C}(q_1,\ldots,q_r)[X]
    \end{align*}	    
    with rational functions
    \begin{align*}
      A_{-\omega_i + \omega_{i+1},1}(q)
        =
	q_{i+1}^r \prod\limits_{\genfrac{}{}{0pt}{}{j=0}{j \neq i}}^r
	\frac{q_j}{q_j q_{i+1} - q_{j+1} q_i}
	.
    \end{align*}
  \end{cor}	  

  \begin{proof}
    The weights of the fundamental representation \( V(\omega_1) \) are
    \( \omega_1, -\omega_1 +\omega_2 , \ldots , -\omega_{n-1} + \omega_n , -\omega_n \) all
    of multiplicity \( 1 \). Again, by Theorem \ref{thm:character-formula} 
    the claimed character formula follows and
    \begin{align*}
      A_{-\omega_i + \omega_{i+1},1}(q)
        & =
	  \frac{(-1)^0}{0! q^{0(-\omega_i + \omega_{i+1})}} \frac{d^0}{(dz)^0}
	  \left[
	    \prod\limits_{\genfrac{}{}{0pt}{}{j=0}{j \neq i}}^r
	    \frac{1}{1-q_j^{-1} q_{j+1}^{} z}
	  \right]_{z=q_i^{} q_{i+1}^{-1}}
      \\	
        & =
	  \prod\limits_{\genfrac{}{}{0pt}{}{j=0}{j \neq i}}^r
  	  \frac{1}{1-q_j^{-1} q_{j+1}^{} q_i^{} q_{i+1}^{-1}}
      \\	
        & =
	  \prod\limits_{\genfrac{}{}{0pt}{}{j=0}{j \neq i}}^r
	\frac{q_j q_{i+1}}{q_j q_{i+1} - q_{j+1} q_i}
      \\	
        & =
	  q_{i+1}^r \prod\limits_{\genfrac{}{}{0pt}{}{j=0}{j \neq i}}^r
	  \frac{q_j}{q_j q_{i+1} - q_{j+1} q_i}
      .
    \end{align*}
  \end{proof}  

  \begin{rem}
    \label{rem:dh-measure}
    For \( \mathfrak{g} = \mathfrak{sl}(r+1,\mathbb{C}) \) we have an interesting aspect
    coming up. Since \( S^N V(\omega_1) = V(N\omega_1) \), it is interesting to ask how the
    formulas obtained in Corollary \ref{cor:character-formula-slr+1_fundamental} compare to the
    asymptotic theory of the Duistermaat-Heckman measure with respect to the sequence of
    representations \( V(N\omega_1) \). 
  \end{rem} 

  \begin{nte}
    Similarly to Corollary \ref{cor:character-formula-slr+1_fundamental} one can compute
    the characters of the symmetric powers of the representations \( V(\omega_i)\) for
    \(i=2,\ldots,r\). Note that although the number of weights contributing to
     \( S^N V(\omega_1) = V(N\omega_1) \) grows with \(N\), their multiplicities always remain
     equal to \(1\). Nevertheless, the rational functions associated to \(V(\omega_1)\)
     do not carry only trivial information, the number \(1\), but also encode which
     weights appear in \(S^N V(\omega_1) \).
     In contrast, the weight multiplicities in \( S^N V(\omega_i) \) for \(i = 2,\ldots,r \)
     are non-trivial and consequently their associated rational functions encode much more
     information. It is part of the full version of this extended abstract to compute
     the characters of the \( S^N V(\omega_i) \) and compare those formulas.
  \end{nte}	  

  For representations with higher dimensional weight spaces (\( \dim \geq 2 \))
  the computations become more difficult. We will demonstrate this by an example. 

  \begin{exl}
    Let \( \mathfrak{g} = \mathfrak{sl}(3,\mathbb{C}) \) and \( V(\omega_1 + \omega_2) \)
    be its adjoint representation which decomposes as shown in the following picture, where
    \( q = e^{i \langle \cdot , x \rangle } = (q_1,q_2)=(a,b) \) with respect
    to the fundamental weights \( \omega_1,\omega_2 \) and the simple coroots
    \( \alpha_1^{\vee}, \alpha_2^{\vee} \).
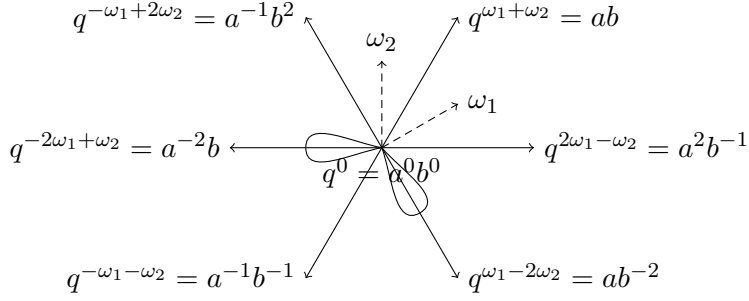
    \[
    \begin{tikzpicture}
      \draw[->] (0,0) -- (1,1.73) node[right] {\( q^{\omega_1 + \omega_2} = a b \)}; 
      \draw[->] (0,0) -- (2,0) node[right] {\( q^{2 \omega_1 - \omega_2} = a^2 b^{-1} \)};
      \draw[->] (0,0) -- (1,-1.73) node[right] {\( q^{\omega_1 - 2 \omega_2} = a b^{-2} \)};
      \draw[->] (0,0) -- (-1,-1.73) node[left] {\( q^{-\omega_1 - \omega_2} = a^{-1} b^{-1} \)};
      \draw[->] (0,0) -- (-2,0) node[left] {\( q^{-2\omega_1 + \omega_2} = a^{-2} b \)};
      \draw[->] (0,0) -- (-1,1.73) node[left] {\( q^{-\omega_1 + 2 \omega_2} = a^{-1} b^2 \)};
      \draw[-] (0,0) to [out=170,in=90] (-1,0);
      \draw[-] (-1,0) to [out=-90,in=-170] (0,0) node[below] {\( q^0 = a^0 b^0 \)};
      \draw[-] (0,0) to [out=-50,in=30] (0.5,-0.865);
      \draw[-] (0.5,-0.865) to [out=-150,in=-70] (0,0);

      \draw[->, densely dashed] (0,0) -- (1,0.58) node[right] {$\omega_1$};
      \draw[->, densely dashed] (0,0) -- (0,1.15) node[above] {$\omega_2$};
    \end{tikzpicture}
    \]
    The picture shows the Littelmann paths \( \mathcal{P}_{\omega_1 + \omega_2} \) of
    shape \( \omega_1 + \omega_2 \) (see e.g. \cite{MR1253196}) and the elements of
    \( \mathbb{Z}[a^{\pm 1},b^{\pm 1}] \) corresponding to the weights of
    \( V(\omega_1 + \omega_2) \).  Here the difficulty lies in computing the rational
    function associated to the
    zero weight which has multiplicity \(2\). This is a first example of a non-trivial
    polynomial \( p_k(N) \) coming up, namely \( p_2(N) = N+1 \). We have
    \begin{align*}
      \Char V(N \omega_1)(ix) = 
	& (A_{0,1}(q) + A_{0,2}(q)p_2(N)) q^{N0}
      \\
	& + A_{2 \omega_1 - \omega_2 ,1}(q) q^{N(2 \omega_1 - \omega_2)}
          + A_{-2 \omega_1 + \omega_2 ,1}(q) q^{N(-2 \omega_1 + \omega_2)}
      \\
	& + A_{\omega_1 -2 \omega_2 ,1}(q) q^{N(\omega_1 -2 \omega_2)}
	  + A_{-\omega_1 +2 \omega_2 ,1}(q) q^{N(-\omega_1 +2 \omega_2)}
      \\
	& + A_{-\omega_1 - \omega_2 ,1}(q) q^{N(-\omega_1 - \omega_2)}
	  + A_{\omega_1 + \omega_2 ,1}(q) q^{N(\omega_1 + \omega_2)}
      .
    \end{align*}
    The difficult part is
    \begin{align*}
      A_{0,1}(q)
        & = \frac{d}{dz} \left[
	    \prod\limits_{\nu \in X \setminus 0}
	    \frac{1}{(1-q^{\nu}z)^{m_{\omega_1 + \omega_2}(\nu)}}
	    \right]_{z = q^0 = a^0 b^0 = 1}
      \\
        & = \frac{-3a^4b^4}{(ab-1)^2(a-b^2)^2(a^2-b)^2}
    \end{align*}
    and we obtain
    \begin{align*}
      \Char V(N \omega_1)(ix) = 
	& \frac{-(3a^4b^4 + a^4b^4p_2(N)) \cdot a^0 b^0}{(ab-1)^2(a-b^2)^2(a^2-b)^2}
      \\
	& + \frac{a^{16}b \cdot a^{2N} b^{-N}}{(ab-1)(a-b^2)(a^2-b)^2(a^3-1)(a^3-b^3)(a^4-b^2)} 
      \\
	& + \frac{-b^{9} \cdot a^{-2N} b^{N}}{(ab-1)(a-b^2)(a^2-b)^2(a^3-1)(a^3-b^3)(a^4-b^2)} 
      \\
	& + \frac{a^{9} \cdot a^{N} b^{-2N}}{(ab-1)(a-b^2)^2(a^2-b)(b^3-1)(a^3-b^3)(a^2-b^4)} 
      \\
	& + \frac{-ab^{16} \cdot a^{-N} b^{2N}}{(ab-1)(a-b^2)^2(a^2-b)(b^3-1)(a^3-b^3)(a^2-b^4)} 
      \\
	& + \frac{ab \cdot a^{-N} b^{-N}}{(ab-1)^2(a-b^2)(a^2-b)(b^3-1)(a^3-1)(a^2b^2-1)} 
      \\
	& + \frac{-a^{9}b^{9} \cdot a^{N} b^{N}}{(ab-1)^2(a-b^2)(a^2-b)(b^3-1)(a^3-1)(a^2b^2-1)} 
      .	
  \end{align*}	    
  \end{exl}

  Let us end this section with an important note.

  \begin{nte}
    \label{nte:Iterated-pfd-Fourier-series}
    In the notation of Theorem \ref{thm:character-formula} note that iterated
    partial fraction decomposition with respect to the variables \(q_1,\ldots,q_r\)
    gives the Fourier series associated to the character of \(S^N V(\lambda)\).
    Thus, decomposing the character formula in Theorem \ref{thm:character-formula}
    further with respect to \(q_1,\ldots,q_r\) yields the weight multiplicity functions
    \(m_{\lambda,N}\). We illustrate this by elaborating on Example
    \ref{exl:character-formula_sl2_V2} where \(r=1\). That is, let us decompose the character
    \begin{align*}
      \Char S^N V(2)(ix) =
        & 
	\frac{q^6}{(q^4-1)(q^2-1)} \cdot q^{2N} + \frac{-q^2}{(q^2-1)^2} \cdot q^{0} \\
        & + \frac{1}{(q^4-1)(q^2-1)} \cdot q^{-2N}
    \end{align*}
    further with respect to \(q\). For e.g. \(N=0,\ldots,5\) this gives
    \begin{align*}
      N & & \Char S^N V(2)(ix) \\
      1 & & q^2 + 1 + q^{-2} \\
      2 & & q^4 + q^2 + 2 + q^{-2} + q^{-4}\\
      3 & & q^6 + q^4 + 2q^2 + 2 + 2q^{-2}+ q^{-4} + q^{-6} \\
      4 & & q^8 + q^6 + 2q^4 + 2q^2 + 3 + 2q^{-2} + 2q^{-4} + q^{-6} + q^{-8} \\
      5 & & q^{10} + q^8 + 2q^6 + 2q^4 + 3q^2 + 3 + 3q^{-2} + 2q^{-4} + 2q^{-6} + q^{-8} + q^{-10}
      .
    \end{align*}
  \end{nte}	  

  \section{A residue-type generating function for the weight multiplicities}
  \label{sec:residue-type-generating-function-multiplicities}
  
  Consider the Fourier series associated to the character of the representation
  \( S^N V(\lambda) \) of our Lie algebra \( \mathfrak{g} \):
  \begin{align*}
    \Char S^N V(\lambda) (ix)
      = \sum\limits_{\nu \in X} m_{\lambda,N}(\nu) e^{i \langle	\nu , x \rangle}
      .
  \end{align*}
  Here \( m_{\lambda,N} \) denotes the weight multiplicity function of
  \(S^N V(\lambda)\). Then, by inverse Fourier transform we can recover the
  Fourier coefficients \( m_{\lambda,N}(\nu) \) as
  \begin{align*}
    m_{\lambda,N}(\nu)
      =
      \frac{1}{(2 \pi)^r}
      \int\limits_{\mathfrak{h}_{\mathbb{R}}/2 \pi X^{\ast}}
      e^{-i \langle \nu , x \rangle} \Char S^N V(\lambda)(ix)
      dx
      .
  \end{align*}  
  Here \(dx\) is Lebesgue measure on \(\mathfrak{h}_{\mathbb{R}}\)
  normalized such that the volume of the torus \(T^r =
  \mathfrak{h}_{\mathbb{R}}/2 \pi X^{\ast}\) is \((2\pi)^r\). Note that
  \(r\) is the rank of \(\mathfrak{g}\). This yields the generating function for the
  weight multiplicity functions \(m_{\lambda,N} \) evaluated at a specific
  weight. That is,

  \begin{prp}
    \label{prp:genf-multiplicities}
    Let \( \mathfrak{g} \) be a semi-simple complex Lie algebra of rank \(r\) and
    \(V(\lambda)\) a fixed irreducible representation of \(\mathfrak{g}\). Let
    \( m_{\lambda,N}\) be the weight multiplicity function of the \(N\)-th
    symmetric power \(S^N V(\lambda)\). Let \(\mu \in X\) be a fixed weight.
    Then, the formal power series \(\sum_{N=0}^{\infty} z^N m_{\lambda,N}(\mu)\)
    is a holomorphic function in the variable \(z\) on \(|z|\leq R <1\).
    Moreover, we have the identity
    \begin{align}
      \label{eq:genf-sym}
      \sum\limits_{N=0}^{\infty} z^N m_{\lambda,N}(\mu)
	= \frac{1}{(2 \pi)^r} \int\limits_{T^r}
	  e^{-i \langle \mu , x \rangle} \prod\limits_{\nu \in X}
	  \frac{1}{(1-e^{i \langle \nu, x\rangle}z)^{m_{\lambda,1}(\nu)}}
	  dx
	  .
    \end{align}  
  \end{prp}

  \begin{proof}
    The assertion follows from the fact that the dimension of the symmetric
    powers of a representation grows sub-exponentially in \(N\) as
    \begin{align*}
      \dim S^N V(\lambda)
	=
	\binom{\dim V(\lambda)-1+N}{N}
	.
    \end{align*}  
    This amounts to say that for the fixed weight \(\mu \in X\) the power series
    \(\sum_{N=0}^{\infty} z^N e^{-i\langle \mu,x \rangle}\Char S^N V(\lambda) (ix)\)
    converges absolutely on \(|z|\leq R<1\) uniformely in \(x \in T^r\).
    Namely, for arbitrary such \(x\) we have
    \begin{align*}
      |e^{-i\langle \mu,x \rangle}\Char S^N V(\lambda) (ix)|
	& =
	  | \sum\limits_{\nu \in X} m_{\lambda,N}(\nu)
	    e^{i\langle \nu,x \rangle} |
      \\
      \mbox{\tiny (triangle inequality)}
	& \leq
	  \sum\limits_{\nu \in X} m_{\lambda,N}(\nu)
      \\
	& =
	  \binom{\dim V(\lambda)-1+N}{N}
      \\
      \mbox{\tiny (\(C\) some constant)}
	& =
	  C N^{\dim V(\lambda)-1} + \mbox{\small lower terms}
      .	  
    \end{align*}  
    Therefore the radius of convergence is given by
    \begin{align*}
      r
	& =
	  \frac{1}{
	    \limsup\limits_{N\rightarrow\infty}
	    \sqrt[N]{|e^{-i\langle \mu,x \rangle}\Char S^N V(\lambda) (ix)|}
	  }
      \\
	& =
	  \frac{1}{
	    \limsup\limits_{N\rightarrow\infty}
	    \sqrt[N]{|C N^{\dim V(\lambda)-1} + \mbox{\tiny lower terms}|}
	  }
      \\
	& =
	  1
	  .
    \end{align*}  
    By Lemma \ref{lem:Molien-formula} the right-hand side of Equation
    \eqref{eq:genf-sym} equals
    \begin{align*}
      \frac{1}{(2\pi)^r} \int\limits_{T^r} e^{-i\langle \mu,x \rangle}
      \sum\limits_{N=0}^{\infty} z^N \Char S^N V(\lambda) (ix) dx
    \end{align*}  
    and since the previous convergence arguments allow us to integrate  term
    by term, this finishes the proof.
  \end{proof}  
  
  Now we are able to explain why the generating function in Equation \eqref{eq:genf-sym}
  is of residue-type. 

  \begin{cor}[Residue-type]
  \label{cor:multiplicities-residue-type}
    Let \( \mathfrak{g} \) be a semi-simple complex Lie algebra of rank \(r\) and
    \(V(\lambda)\) a fixed irreducible representation of \(\mathfrak{g}\). Let
    \( m_{\lambda,N}\) be the weight multiplicity function of the \(N\)-th
    symmetric power \(S^N V(\lambda)\). Let \(\mu \in X\) be a fixed weight and
    denote \(q^{\mu} = e^{i \langle \mu , x \rangle}\) as above.
    Then,
    \begin{align}
      m_{\lambda,N}(\mu)
	= \frac{1}{(2 \pi)^r} \int\limits_{T^r}
	  q^{- \mu}
	  \sum\limits_{\nu \in X} q^{N\nu}
	  \sum\limits_{k=1}^{m_{\lambda}(\nu)}
          A_{\nu,k}(q) \cdot p_k (N)
	  dx
	  .
    \end{align}
    In particular, the multiplicity \( m_{\lambda,N}(\mu) \) equals the constant term of the function
    \( \Char S^N V(\lambda) (ix) \) shifted by \( q^{-\mu} \).
  \end{cor}

  \begin{proof}  
    This is a direct consequence of Proposition \ref{prp:genf-multiplicities} and Theorem
    \ref{thm:character-formula}. 
  \end{proof}  

  \begin{exl}
    In the case \( \mathfrak{g} = \mathfrak{sl}(2,\mathbb{C}) \) and the symmetric powers
    \( S^N V(2) \) of the adjoint representation we have described in Note
    \ref{nte:Iterated-pfd-Fourier-series} that, e.g. for \( N = 4 \),
    \begin{align*}
      \Char S^4 V(2)(ix) 
        = q^8 + q^6 + 2q^4 + 2q^2 + 3 + 2q^{-2} + 2q^{-4} + q^{-6} + q^{-8}
      .
    \end{align*}
    Note that \( q = e^{ix} \).
    Now, in view of Corollary \ref{cor:multiplicities-residue-type}, the multiplicity of
    the weight \(\mu=2\omega_1\) in \( S^4 V(2) \) is given by 
    \begin{align*}
      m_{\lambda,N}(\mu)
        & = m_{2,4}(2)
      \\
        \mbox{\tiny \((q=e^{ix})\)}
        & = \frac{1}{2 \pi}
	    \int\limits_{S^1}
	    q^{-2} (q^8 + q^6 + 2q^4 + 2q^2 + 3 + 2q^{-2} + 2q^{-4} + q^{-6} + q^{-8})
	    dx
      \\
        & = \frac{1}{2 \pi}
	    \int\limits_{S^1}
	    q^6 + q^4 + 2q^2 + 2q^0 + 3q^{-2} + 2q^{-4} + 2q^{-6} + q^{-8} + q^{-10}
	    dx
      \\
        & = \frac{1}{2 \pi}
	    \int\limits_{S^1} 2q^0 dx
      \\
        & = 2
    \end{align*}	    
  \end{exl}	  

  \begin{rem}
    Similar to Proposition \ref{prp:genf-multiplicities} we have a generating function for
    the weight multiplicity functions \( m_{\lambda,N}^{\Lambda} \) of the exterior powers
    \(\Lambda^NV(\lambda)\) of an irreducible representation \( V(\lambda) \).
    First, realize (see \cite[Chapter 9, §4.3]{MR2265844}) that the graded character
    of the exterior algebra of \(V(\lambda)\) is given by
    \begin{align*}
      \Char \Lambda V(\lambda)
        =
        \sum\limits_{N=0}^{\infty} z^N \Char \Lambda^NV(\lambda)
	=
	\prod\limits_{\nu \in X}
	(1+e^{\nu}z)^{m_{\lambda,1}^{\Lambda}(\nu)}
	.
    \end{align*}
    Then, again by inverse Fourier transform and the same convergence arguments we
    obtain a generating function with radius of convergence equal to \(1\), satisfying
    the identity
    \begin{align*}
      \sum\limits_{N=0}^{\infty} z^N m_{\lambda,N}^{\Lambda}(\mu)
	= \frac{1}{(2 \pi)^r} \int\limits_{T^r}
	  e^{-i \langle \mu , x \rangle} \prod\limits_{\nu \in X}
	  (1+e^{i \langle \nu,x \rangle}z)^{m_{\lambda,1}^{\Lambda}(\nu)}
	  dx
	  .
    \end{align*}  
  \end{rem}

  \begin{rem}
    For the tensor powers \( T^N V(\lambda) \) of a fixed irreducible representation 
    \(V(\lambda)\) with weight multiplicity functions \( m_{\lambda,N}^{T} \) we have
    the identity \( \Char T^N V(\lambda) = (\Char V(\lambda) )^N \) and consequently
    \begin{align*}
      \sum\limits_{N=0}^{\infty} z^N m_N^{T}(\mu)
         = &  \frac{1}{(2\pi)^r} \int\limits_{T^r} e^{-i\langle \mu,x \rangle}
	    \sum\limits_{N=0}^{\infty} z^N \Char T^N V(\lambda) (ix) dx
      \\
        = & \frac{1}{(2\pi)^r} \int\limits_{T^r} e^{-i\langle \mu,x \rangle}
	    \frac{1}{1-\Char V(\lambda)(ix)z} dx
      .
    \end{align*}  
    This constitutes a holomorphic function with radius of convergence equal
    to \(\frac{1}{\dim V(\lambda)}\).
  \end{rem}  

  \section{Connection to vector partition functions}
  \label{sec:vector-partition-functions}

  For an integral matrix \( A \in \mathbb{Z}^{(m,d)} \) with
  \( \ker (A) \cap \mathbb{R}^d_+ = \{0\}\) we define the vector partition function
  \( \phi_A  : \mathbb{Z}^m \rightarrow \mathbb{N} \) by
  \begin{align*}
    \phi_A (b) = \# \{ x \in \mathbb{N}^d : Ax=b\}
    .
  \end{align*}  
  Let \(c_1,\ldots,c_d\) denote the columns of \(A\)
  and use multiexponent notation \( z^b = z_1^{b_1} \cdots z_m^{b_m} \), \(b \in \mathbb{Z}^m\). 
  Then, as stated in \cite[Equation (1)]{arx09121131}, on
  \( \{ z \in \mathbb{C}^m : |z^{c_k}| < 1 \mbox{ for } k = 1,\ldots,d \} \)
  we have the identity                                                                             
  \begin{align*}
    f_A(z) := \sum_{b \in \mathbb{Z}^m} \phi_A (b) z^b = \prod_{k=1}^d \frac{1}{1-z^{c_k}}         
  \end{align*}
  and
  \begin{align*}
    \phi_A(b) = \mbox{const} \left[ f_A(z) \cdot z^{-b} \right]
    .
  \end{align*}
  Now, there is an obvious connection between the graded character of the symmetric algebra
  \(S V(\lambda)\) of an irreducible representation \(V(\lambda)\) of a complex semi-simple Lie algebra
  \( \mathfrak{g} \) and the theory of vector partition functions, which is given by
  Lemma \ref{lem:Molien-formula}. Namely, if \( \mathfrak{g} \) is of rank r, then one has
  a matrix \( A \in \mathbb{Z}^{(r+1,\dim V(\lambda))} \) encoding the weights of \( V(\lambda) \)
  in terms of the coordinate system given by the fundamental weights
  \( \omega_1, \ldots, \omega_r \). This information corresponds to the first \( r \) rows
  of each column of \( A \). In addition to that, we have the \( (r+1) \)-th row which
  associates to the grading given by \( z \) in Lemma \ref{lem:Molien-formula}. That is,
  our particular matrix \( A \) has the following properties
  \begin{enumerate}
    \item The last row of \( A \) equals \( (1, \ldots , 1) \). \hfill (grading)
    \item The columns of \( A \) reflect the Weyl group action. \hfill (symmetry)
    \item The columns of \( A \) appear with multiplicities. \hfill (multiplicity)
  \end{enumerate}
  In contrast to the computational and algorithmic aspects of iterated partial fraction
  decomposition as proposed in \cite{MR2096740} and continued e.g. in \cite{arx09121131}
  for ``arbitrary'' matrices \( A \), our interests are different. They lie in
  investigating further the closed character formulas for the symmetric powers and
  the impact of the grading, symmetry and multiplicity properties of our matrix
  \( A \) on the iterated partial fraction decomposition. One aspect is described
  in detail in Section \ref{sec:Weyl-group-orbits}.

  \section{Weyl group orbits and the Main Theorem}
  \label{sec:Weyl-group-orbits}
  
  In the notation of Theorem \ref{thm:character-formula} write the character of
  \( S^N V(\lambda) \) as the sum over the dominant weights and their Weyl group
  orbits, i.e.
  \begin{align*}
    \Char S^N V(\lambda)(ix)
      = \sum\limits_{\nu \in X^+} \sum\limits_{w \in W/W_{\nu}} q^{N w.\nu}
        \sum\limits_{k=1}^{m_{\lambda}(\nu)}
        A_{w.\nu,k}(q) \cdot p_k (N)
	.
  \end{align*}
  Here \(W_{\nu}\) denotes the stabilizer of the weight \(\nu\). Note
  that the multiplicity of a weight is invariant under the operation of the
  Weyl group (see e.g. \cite[Proposition 10.22]{MR2188930}). Now, for a fixed dominant weight
  \( \nu \in X^+ \) let
  \begin{align}
  \label{eq:character-dominant-weight-summand}
    f_{\nu,N}(q) 
      = \sum\limits_{w \in W/W_{\nu}} q^{N w.\nu}
        \sum\limits_{k=1}^{m_{\lambda}(\nu)}
        A_{w.\nu,k}(q) \cdot p_k (N)
	\in \mathbb{C}(q_1,\ldots,q_r)[X]
  \end{align}
  so that
  \begin{align*}
    \Char S^N V(\lambda)(ix) = \sum_{\nu \in X^+} f_{\nu,N}(q)
    .
  \end{align*}
  It is interesting to ask how the iterated partial fraction decomposition with respect
  to the variables \( q_1,\ldots,q_r \) of a single summand \( f_{\nu,N}(q) \)
  looks like. Examples indicate that this decomposition of \( f_{\nu,N}(q) \)
  does not yield information about the weights outside the convex hull of the
  Weyl group orbit \( W.(N\nu) \). Furthermore, some additional terms appear 
  which sum up to zero when taken over all dominant weights \(X^+\). We will
  illustrate this by an example in the case of \( \mathfrak{g} \) being of
  rank \(1\) to avoid confusing computations.
  \begin{exl}
    Consider the sequence of representations \( S^N V(3) \) of
    \(\mathfrak{g}=\mathfrak{sl}(2,\mathbb{C})\).
    Then, by Corollary \ref{cor:character-formula-sl2_irreducible} we have
    \begin{align*}
      \Char S^N V(3)(ix) =
        & \frac{q^{12} \cdot q^{3N}}{(q^6-1)(q^4-1)(q^2-1)} 
	  + \frac{-q^6 \cdot q^{N}}{(q^4-1)(q^2-1)^2} 
      \\
        & 
	+ \frac{q^2 \cdot q^{-N}}{(q^4-1)(q^2-1)^2} 
	  + \frac{- q^{-3N}}{(q^6-1)(q^4-1)(q^2-1)}
    \end{align*}
    where \( q = e^{ix} \). Hence, following the notation introduced in Equation
    \eqref{eq:character-dominant-weight-summand} we set
    \begin{align*}
      f_{1,N}(q) & = \frac{-q^6 \cdot q^{N}}{(q^4-1)(q^2-1)^2} 
                     + \frac{q^2 \cdot q^{-N}}{(q^4-1)(q^2-1)^2} \\
      f_{3,N}(q) & = \frac{q^{12}\cdot q^{3N}}{(q^6-1)(q^4-1)(q^2-1)}
                     + \frac{- q^{-3N}}{(q^6-1)(q^4-1)(q^2-1)}
		     .
    \end{align*}	    
    Now, e.g. for \(N=4\), we obtain
    \begin{align*}
      \PFD_q(f_{1,4}(q))
        = & -q^2 -2 -q^{-2}
      \\
          & -\frac{3}{4(q-1)^2} + \frac{3}{4(q+1)} - \frac{3}{4(q-1)} - \frac{3}{4(q+1)^2}
    \end{align*}
    and
    \begin{align*}
      \PFD_q(f_{3,4}(q))
        = & q^{12} + q^{10} + 2q^8 +3q^6 + 4q^4 +5q^2 +7 \\
	  & +5q^{-2} +4q^{-4} + 3^{-6} +2q^{-8} +q^{-10} +q^{-12} \\
          & +\frac{3}{4(q-1)^2} - \frac{3}{4(q+1)} + \frac{3}{4(q-1)} + \frac{3}{4(q+1)^2}
	  ,
    \end{align*}
    where in each individual decomposition the last four summands are the additional
    terms which sum up to zero. Unfortunately this example indicates that we cannot expect
    a positive formula for the weight multiplicities of the symmetric powers.
  \end{exl}

  \newpage
  \bibliography{character}
  \bibliographystyle{alpha}	

\end{document}